\documentclass[a4paper, 10pt]{amsart}

\usepackage[square, numbers, comma]{natbib} 

\usepackage[usenames,dvipsnames,svgnames,table]{xcolor}
\usepackage{amsmath,amsthm,amssymb,amssymb,esint,verbatim,tabularx,graphicx}
\usepackage{bbm}
\usepackage{fancyhdr}
\usepackage{enumerate}
\usepackage{amsmath}
\usepackage{fancyhdr}
\usepackage{epic}
\usepackage{pgf,tikz}
\usetikzlibrary{arrows}

\usepackage[utf8]{inputenc}
\usepackage{color}
\usepackage{hyperref}
\usepackage{verbatim}
\usepackage{pdfpages}
\usepackage{empheq}

\hypersetup{urlcolor=blue, colorlinks=true} 

\usepackage[inner=2.5cm,outer=2.5cm,bottom=4cm,top=4cm]{geometry}

\newtheorem{theorem}{Theorem}[section]

\newtheorem{lemma}[theorem]{Lemma}
\newtheorem{proposition}[theorem]{Proposition}
\newtheorem{corollary}[theorem]{Corollary}
\newtheorem{definition}{Definition}[section]

\theoremstyle{remark}
\newtheorem{remark}[theorem]{Remark}
\theoremstyle{definition}

\numberwithin{equation}{section}

\newcommand{\R}{\ensuremath{\mathbb{R}}}
\newcommand{\N}{\ensuremath{\mathbb{N}}}

\newcommand{\Z}{\ensuremath{\mathbb{Z}}}

\newcommand{\Div}{\mbox{div}}

\newcommand{\plap}{\ensuremath{\Delta_p}}

\newcommand{\dd}{\,\mathrm{d}}

\DeclareMathOperator{\supp}{supp}

\newcommand{\Grid}{\mathcal{G}_h}

\begin{document}

\title[Finite differences for the parabolic $p$-Laplacian]{Finite difference schemes for the parabolic $p$-Laplace equation}

\author[F.~del~Teso]{F\'elix del Teso}

\address[F. del Teso]{Departamento de Matematicas, Universidad Aut\'onoma de Madrid (UAM), Campus de Cantoblanco, 28049 Madrid, Spain} 

\email[]{felix.delteso\@@{}uam.es}

\urladdr{https://sites.google.com/view/felixdelteso}

\keywords{$p$-Laplacian, mean value property, viscosity solutions, finite differences, explicit scheme.}

\author[E.~Lindgren]{Erik Lindgren}

\address[E. Lindgren]{Department of Mathematics, Uppsala University, Box 480, 751 06 Uppsala, Sweden}
\email[]{erik.lindgren\@@{}math.uu.se}
\urladdr{https://sites.google.com/view/eriklindgren}

\address[]{}
\email[]{}
\urladdr{}

\subjclass[2020]{
35K10, 
35K55, 
35K92, 
35K67, 
 35D40, 
 35B05, 
 49L20. 
 }

\begin{abstract}\noindent   
We propose a new finite difference scheme for the degenerate parabolic equation
\[
\partial_t u - \Div(|\nabla u|^{p-2}\nabla u) =f, \quad p\geq 2.
\]
Under the assumption that the data  is   H\"older continuous, we establish the convergence of the explicit-in-time scheme for the Cauchy problem  provided   a suitable stability type CFL-condition. An important advantage of our approach, is that the CFL-condition makes use of the regularity provided by the scheme to reduce the computational cost.  In particular,  for Lipschitz data, the  CFL-condition  is of the same order as for the heat equation and  independent of  $p$. 
\end{abstract}

\maketitle

\tableofcontents 

\section{Introduction} \label{sec:intro}
Recently, a new monotone finite difference discretization of the $p$-Laplacian was introduced by the authors in \cite{dTLi22}. It is based on the mean value property presented in \cite{dTLi20, BS18}. The aim of this paper is to propose  an explicit-in-time finite difference numerical scheme  for the following Cauchy problem 
\begin{equation}\label{eq:ParabProb}
\begin{cases}
\partial_t u(x,t)-\plap u(x,t) =f(x), & x\in \R^d\times{(0,T)},\\
u(x,0)=u_0(x), & x \in \R^d,
\end{cases}
\end{equation}
and study its convergence. Here, $p\geq 2$ and $\Delta_p$ is the $p$-Laplace operator, 
\[
\plap \psi=\Div(|\nabla \psi|^{p-2}\nabla \psi). 
\]
 The  main result is the pointwise convergence of our scheme given H\"older continuous data ($f$ and $u_0$) and a stability type  CFL-condition. See Theorem \ref{theo:main} for the precise statement and \eqref{as:CFL} for the  CFL-condition. One of the advantages of our approach is that the CFL-condition makes use of the regularity provided by the scheme. As a consequence,  for Lipschitz continuous data, the CFL-condition is  of the same order as the one for the heat equation.   In general, the order of the CFL-condition depends on $p$ and on the regularity of the data.

\subsection{Related literature}  Equation \eqref{eq:ParabProb} has attracted much attention in the last decades. We refer to \cite{Db} and \cite{DGV} for the theory for weak solutions of this equation and to \cite{JLM} for the relation between viscosity solutions and weak solutions. To the best of our knowledge,  the best regularity results known are $C^{1,\alpha}-$regularity in space for some $\alpha>0$ (see \cite[Chapter IX]{Db}) and $C^{0,1/2}-$regularity in time (see \cite[Theorem 2.3]{Bo}).

The literature regarding finite difference schemes for parabolic problems involving the $p$-Laplacian is quite scarce. One reason for that is naturally that, since the $p$-Laplacian is in divergence form, it is very well suited for methods based on finite elements, see for instance \cite{BL94,J00,DCR07,AGW04,FedPP-L} for related results.

In the stationary setting, there has been some development of finite difference methods the past 20 years. Section 1.1 in \cite{ObermanpLap} provides an accurate overview of such results, we will only mention a few. In \cite{CoLeMa17,dTMP18,FFGS13,ObermanpLap}, finite difference schemes for the $p$-Laplace equation based on the mean value formula for the \emph{normalized} $p$-Laplacian (cf. \cite{MPR12a}) are considered.  Since the  corresponding parabolic  equation for the normalized $p$-Laplacian is completely different in nature (see \cite{JuKa06,Ker11}), these methods do not seem very well suited to be used for the parabolic equation considered in this paper. In  \cite{dTLi22}, the authors of the present paper studied a monotone finite difference discretization of the $p$-Laplacian based  on the  mean value property presented in \cite{dTLi20, BS18}. We also seize the opportunity to mention \cite{Obe05}, where difference schemes for degenerate elliptic and parabolic equations (but not for equation \eqref{eq:ParabProb}) are discussed.

 It is noteworthy that, in dimension $d=1$,  the spatial derivative of a solution of \eqref{eq:ParabProb} is a solution of the  Porous Medium Equation (PME). See \cite{Vaz06, Vaz07} for a general presentation  of the PME, and \cite{IaSaVa08} for a proof of  this fact. Finite difference schemes for the PME are well known,  see \cite{DiHo84,EmSi12,Mon16,dTEnJa18,dTEnJa19}.

\section{Assumptions and main results}
In this section, we introduce a general form of finite difference discretizations of $\plap$ and the associated numerical scheme for \eqref{eq:ParabProb}. This is followed by our assumptions, the notion of solutions for \eqref{eq:ParabProb} and the formulation of our main result.

\subsection{Discretization and scheme}

In order to treat \eqref{eq:ParabProb}, we consider a general discretization of $\Delta_p$ of the form
\begin{equation}\label{eq:GenDisc}
D^h_p\psi(x)=\sum_{y_\beta\in \Grid} J_p(\psi(x+y_\beta)-\psi(x)) \omega_{\beta},
\end{equation}
where 
$$J_p(\xi)=|\xi|^{p-2}\xi,\quad \xi \in \R,\quad  \Grid:=h\Z^d=\{y_\beta:= h \beta \, : \, \beta \in \Z^d\}$$
 and $\omega_\beta$ are certain weights $\omega_\beta=\omega_\beta(h)$ satisfying $\omega_\beta=\omega_{-\beta}\geq0$.

We also need to introduce a time discretization. We will employ an explicit and uniform-in-time discretization. Let $N\in \N$ and consider a discretization parameter $\tau>0$ given by
$\tau =T/N$. Consider also the sequence of times $\{t_j\}_{j=0}^{N}$ defined by $t_0=0$ and $t_j=t_{j-1}+ \tau= j\tau$. The time grid, $\mathcal{T}_\tau$, is given by
\[
\mathcal{T}_{\tau}= \bigcup_{j=0}^N \{t_j\}.
\]
Then, our general form of an explicit finite difference scheme of \eqref{eq:ParabProb} is given by 
\begin{equation}\label{eq:numsch}
\begin{cases}
U^j_\alpha= U_\alpha^{j-1} +\tau\left(D_p^h U_{\alpha}^{j-1}+f_\alpha\right), & \alpha \in \Z^d,\,  j=1,\ldots,N,\\
U^0_\alpha=(u_0)_\alpha & \alpha \in \Z^d,
\end{cases}
\end{equation}
where $f_\alpha:=f(x_\alpha)$, $(u_0)_\alpha=u_0(x_\alpha)$ and $D_p^h$ is given by \eqref{eq:GenDisc}. 

\subsection{Assumptions}

In order to ensure convergence of the scheme \eqref{eq:numsch}, we impose the following hypotheses on the data and the discretization parameters. This entails a regularity assumption on the data, some assumptions on the discretization and a nonlinear CFL-condition on the parameters, as is customary for explicit schemes.

\medskip
\noindent\textbf{Hypothesis on the data. } We assume that
\begin{equation}\label{as:u0f}\tag{$\textup{A}_{u_0,f}$}
\textup{$u_0, f:\R^d \to \R$ are bounded and globally H\"older continuous functions for some $a\in(0,1]$.
}
\end{equation}
More precisely, 
\[
|u_0(x)-u_0(y)|\leq L_{u_0}|x-y|^a  \quad \textup{and} \quad |f(x)-f(y)|\leq L_{f}|x-y|^a, \qquad \textup{for all} \quad x,y\in \R^d,
\]
for some constants $L_{u_0},L_{f}\geq0$. Sometimes we will write 
$\Lambda_{u_0}(\delta):=L_{u_0}\delta^a$ and $\Lambda_{f}(\delta):=L_{f}\delta^a
$
to simplify the presentation.

\medskip
\noindent\textbf{Hypothesis on the spatial discretization.}
 For the discretization, we assume the following type of monotonicity and boundedness:
\begin{equation}\label{as:disc}\tag{$\textup{A}_\omega$}
 \textup{$\omega_\beta=\omega_{-\beta}\geq0$, $w_\beta=0$ for $y_\beta \not\in B_r$ for some $r>0$, and $\sum_{y_\beta\in \Grid}\omega_\beta\leq  M r^{-p}$ } 
 \end{equation}
   Here  $M=M(p,d)>0$. 
In addition, we assume the following consistency for the discretization:
\begin{equation}\label{as:cons}\tag{$\textup{A}_{c}$}
\textup{For $\psi \in C^2_b(\R^d\times[0,T])$, we have that $D^h_p\psi=\plap \psi+o_h(1)$  as $h\to0^+$ uniformly in $(x,t)$.}
\end{equation}

Examples of discretizations satisfying these properties can be found in Section \ref{sec:discretizations}.

\medskip
\noindent\textbf{Hypothesis on the discretization parameters.}
We assume the following stability condition on the numerical parameters:
 \begin{equation}\label{as:CFL}\tag{$\textup{CFL}$}
h=o_r(1) \quad \textup{and} \quad  \tau\leq  C r^{2+(1-a)(p-2)}
\end{equation}
with
\[
C=\min\left\{1, \frac{1}{M (p-1)\left(L_{u_0}+TL_f +3\tilde{K}+1\right)^{p-2}}\right\}
\]
 and $\tilde{K}$ a constant given in \eqref{eq:ctefam}, depending on $p$, the modulus of continuity in time of the discretized solution and some universal constants coming from a mollifier.

 \begin{remark} For Lipschitz data $u_0$ and $f$,  the condition \eqref{as:CFL} reads $\tau \leq{Cr^2}$ for a certain constant $C=C(u_0,f,d,p,T)>0$.  We note that, regardless of the constant $C$, the relation between $\tau$ and $r$ is always quadratic (as in the linear case $p=2$) and independent of $p$. It is important to mention that this is computationally very relevant, especially if we want to deal with problems related to large $p$.
\end{remark}

\subsection{Main result} We  now  state our main result regarding the convergence of the scheme.  Several other properties of the scheme are also obtained, but we will state them later. 

\begin{theorem}\label{theo:main}
Let $p\in[2,\infty)$  and assume \eqref{as:u0f} and \eqref{as:disc}. Then
 for every $h,\tau>0$, there exists a unique solution $U\in \ell^\infty(\Grid\times \mathcal{T}_\tau)$ of \eqref{eq:numsch}.
If in addition, \eqref{as:CFL} and \eqref{as:cons} hold, then  
\[
\max_{(x_\alpha,t_j)\in \Grid \times \mathcal{T}_\tau}|U_\alpha^j- u(x_\alpha,t_j)|\to 0 \quad \textup{as} \quad h\to0^+,
\]
where $u$ is the unique viscosity solution of \eqref{eq:ParabProb}.
\end{theorem}

\subsection{Viscosity solutions} Throughout the paper, we will use the notion of viscosity solutions. For completeness, we define the concept of viscosity solutions of \eqref{eq:ParabProb}, adopting the definition in \cite{JLM}.

\begin{definition}
Assume \eqref{as:u0f}. We say that a bounded lower (\textup{resp.} upper) semicontinuous function $u$ in $\R^d\times[0,T]$ is a \textup{viscosity supersolution} (\textup{resp.} \textup{subsolution}) of \eqref{eq:ParabProb} if 
\begin{enumerate}[(a)]
\item $u(x,0)\geq u_0(x)$  (resp. $u(x,0)\leq u_0(x)$);
\item whenever $(x_0,t_0)\in \R^d\times (0,T)$ and $\varphi\in C^2_b(B_R(x_0)\times(t_0-R,t_0+R))$ for some $R>0$ are such that $\varphi(x_0,t_0)=u(x_0,t_0)$ and $\varphi(x,t)< u(x,t)$  (\text{resp.} $\varphi(x,t)> u(x,t)$) for  $(x,t) \in  B_R(x_0)\times(t_0-R,t_0) $, then we have
\[
\varphi_t(x_0,t_0)-\Delta \varphi(x_0,t_0)\geq f(x_0) \quad (\text{resp.} \quad \varphi_t(x_0,t_0)-\Delta \varphi(x_0,t_0)\leq f(x_0) ).
\]
\end{enumerate}
A \textup{viscosity solution} of \eqref{eq:ParabProb} is a bounded continuous function $u$ being both a viscosity supersolution and a viscosity subsolution \eqref{eq:ParabProb}.
\end{definition}

\begin{remark} We remark that it is not necessary to require strict inequality in the definition above. It is enough to require $\varphi(x,t)\leq u(x,t)$  (\text{resp.} $\varphi(x,t)\geq  u(x,t)$) for  $(x,t) \in  B_R(x_0)\times(t_0-R,t_0) $.
\end{remark}

We also state a necessary uniqueness result that will ensure convergence of the scheme. Without such a result, we would only be able to establish convergence up to a subsequence.  The theorem below is a consequence of the fact that viscosity solutions are weak solutions (see Corollary 4.7 in \cite{JLM}) and that bounded weak solutions are unique (see Theorem 6.1 in \cite{Db}).

\begin{theorem}\label{teo:main2} Assume \eqref{as:u0f}.  Then there is a unique solution of  \eqref{eq:ParabProb}.
\end{theorem}

\section{Properties of the numerical scheme}

In this section we will study properties of the numerical scheme \eqref{eq:numsch}. More precisely, we establish existence and uniqueness for the numerical solution, stability in maximum norm, as well as conservation of the modulus of continuity of the data. 

\subsection{Existence and uniqueness}
 We have the following existence and uniqueness result for the numerical scheme.
 \begin{proposition}
 Assume \eqref{as:u0f}, \eqref{as:disc}, $p\geq2$ and $r,h,\tau>0$. Then there exists a unique solution $U\in \ell^\infty(\Grid \times \mathcal{T}_{\tau})$ of the scheme \eqref{eq:numsch}. 
 \end{proposition}
 \begin{proof}
 First we note that, for a function $\psi\in \ell^\infty(\Grid)$, we have that
 \[
| D^h_p\psi_\alpha|\leq \sum_{y_\beta\in \Grid} J_p(\psi(x_\alpha+y_\beta)-\psi(x_\alpha)) \omega_{\beta} \leq (2\|\psi\|_{\ell^\infty(\Grid)})^{p-1}\sum_{y_\beta\in \Grid}\omega_{\beta}<+\infty.
 \]
Then, for each $\alpha\in \Z$, $U^j_\alpha$ is defined recursively using the values of $U^{j-1}_\beta$ for $\beta\in \Z$, and we have that
\[
\sup_{y_\alpha\in \Grid}|U^j_\alpha|= \sup_{y_\alpha\in \Grid}|U^{j-1}_\alpha| + \tau \left(\left(2\sup_{y_\alpha\in \Grid}|U^{j-1}_\alpha|\right)^{p-1} \sum_{y_\beta\in \Grid}\omega_{\beta} + \sup_{y_\alpha\in \Grid}|f_\alpha|\right).
\]
The conclusion follows since
\[
\sup_{y_\alpha\in \Grid}|f_\alpha|\leq \|f\|_{L^\infty(\R^d)} \quad \textup{and} \quad  \sup_{y_\alpha\in \Grid}|U^0_\alpha|=\sup_{y_\alpha\in \Grid}|u_0(y_\alpha)|\leq \|u_0\|_{L^\infty(\R^d)}.
\]
 \end{proof}

 \subsection{Stability and preservation of the modulus of continuity in space}
 First we will prove that the scheme preserves the regularity of the data.
 \begin{proposition}\label{prop:regusch} Assume \eqref{as:u0f},  \eqref{as:disc}, $p\geq2$,   $r, h, \tau >0$ and \eqref{as:CFL}. Let $U$ be the solution of \eqref{eq:numsch}. For every $j=0,\ldots,N$, we have
 \[
 |U^j_\alpha-U^j_\gamma| \leq \Lambda_{u_0}(|x_\alpha-x_\gamma|) + t_j \Lambda_f(|x_\alpha-x_\gamma|), \quad \textup{for all} \quad x_\alpha,x_\gamma\in \Grid.
 \]
 \end{proposition}
 
 \begin{remark}
 In particular, if both $u_0$ and $f$ are Lipschitz functions with constants $L_{u_0}$ and $L_f$ respectively, the above result reads,
 \[
  |U^j_\alpha-U^j_\gamma| \leq (L_{u_0}+t_j L_f)|x_\alpha-x_\gamma|.
 \]
 \end{remark}
 \begin{proof}[Proof of Proposition \ref{prop:regusch}]
By assumption \eqref{as:u0f}, for any given $x_\alpha,x_\gamma\in \Grid$, we have that
\[
|U^0_\alpha-U^0_\gamma|= |u_0(x_\alpha)-u_0(x_\gamma)|\leq \Lambda_{u_0}(|x_\alpha-x_\gamma|).
\]
Assume by induction that
\[
|U^{j}_\alpha-U^{j}_\gamma|\leq \Lambda_{u_0}(|x_\alpha-x_\gamma|) + t_j \Lambda_f(|x_\alpha-x_\gamma|).
\]
Using the scheme at $x_\alpha$ and $x_\gamma$ we get
\[
U^{j+1}_\alpha-U^{j+1}_\gamma= U^{j}_\alpha-U^{j}_\gamma+\tau \sum_{y_\beta\in \Grid} \left(J_p( U^{j}_{\alpha+\beta}-U^{j}_{\alpha}) - J_p( U^{j}_{\gamma+\beta}-U^{j}_{\gamma})\right) \omega_\beta + \tau (f_\alpha-f_\gamma).
\]
Now, since $p\geq2$, we have, by Taylor expansion, that
\[
J_p( U^{j}_{\alpha+\beta}-U^{j}_{\alpha}) - J_p( U^{j}_{\gamma+\beta}-U^{j}_{\gamma})=(p-1)|\eta_\beta|^{p-2} \left((U^{j}_{\alpha+\beta}-U^{j}_{\gamma+\beta})-(U^{j}_{\alpha}-U^{j}_{\gamma})\right),
\]
for some $\eta_\beta\in \R$ between $(U^{j}_{\alpha+\beta}-U^{j}_{\alpha})$ and $(U^{j}_{\gamma+\beta}-U^{j}_{\gamma})$. Thus,
\begin{equation}\label{eq:pres1}
\begin{split}
U^{j+1}_\alpha-U^{j+1}_\gamma=& (U^{j}_\alpha-U^{j}_\gamma)\left(1-\tau (p-1) \sum_{y_\beta\in \Grid} |\eta_\beta|^{p-2} \omega_\beta \right)\\
&+\tau (p-1) \sum_{y_\beta\in \Grid} |\eta_\beta|^{p-2} (U^{j}_{\alpha+\beta}-U^{j}_{\gamma+\beta})  \omega_\beta  + \tau (f_\alpha-f_\gamma).
\end{split}
\end{equation}
Now observe that, by the induction assumption, we have
\[
|\eta_\beta|\leq \sup_{y_\alpha \in \Grid} \{|U^{j}_{\alpha+\beta}-U^{j}_{\alpha}|\}\leq \sup_{y_\alpha \in \Grid} \{\Lambda_{u_0}(|x_{\alpha+\beta}-x_\alpha|) + t_j \Lambda_f(|x_{\alpha+\beta}-x_\alpha|)\}= \Lambda_{u_0}(| x_{\beta}|) + t_j \Lambda_f( |x_{\beta}|).
\]
By \eqref{as:disc}, we have $w_\beta=0$ for $y_\beta \not\in B_r$ for some $r>0$, and we deduce that
\[
 \sum_{y_\beta\in \Grid} |\eta_\beta|^{p-2} \omega_\beta \leq  \left(\Lambda_{u_0}(r) + t_j \Lambda_f(r)\right)^{p-2}\sum_{y_\beta\in \Grid}\omega_{\beta}\leq \frac{(L_{u_0}+t_j L_f)^{p-2}M}{ r^{2+(1-a)(p-2)}}.
\]
Thus, by \eqref{as:CFL}, we get
\[
\tau (p-1) \sum_{y_\beta\in \Grid} |\eta_\beta|^{p-2} \omega_\beta\leq 1.
\]
Using the above estimate and the induction hypothesis in \eqref{eq:pres1},  we get that
\begin{equation*}
\begin{split}\
|U^{j+1}_\alpha-U^{j+1}_\gamma|\leq& |U^{j}_\alpha-U^{j}_\gamma|\left(1-\tau (p-1) \sum_{y_\beta\in \Grid} |\eta_\beta|^{p-2} \omega_\beta \right)\\
&+\tau (p-1) \sum_{y_\beta\in \Grid} |\eta_\beta|^{p-2} |U^{j}_{\alpha+\beta}-U^{j}_{\gamma+\beta}|  \omega_\beta  + \tau |f_\alpha-f_\gamma|\\
\leq& \left(\Lambda_{u_0}(|x_\alpha-x_\gamma|) + t_j \Lambda_f(|x_\alpha-x_\gamma|)\right) \left(1-\tau (p-1) \sum_{y_\beta\in \Grid} |\eta_\beta|^{p-2} \omega_\beta \right)\\
&+\tau (p-1) \sum_{y_\beta\in \Grid} |\eta_\beta|^{p-2} \left(\Lambda_{u_0}(|x_{\alpha+\beta}-x_{\gamma+\beta}|) + t_j \Lambda_f(|x_{\alpha+\beta}-x_{\gamma+\beta}|)\right) \omega_\beta  \\
&+ \tau \Lambda_f (|x_\alpha-x_\gamma|)\\
\leq& \left(\Lambda_{u_0}(|x_\alpha-x_\gamma|) + t_j \Lambda_f(|x_\alpha-x_\gamma|)\right) \left(1-\tau (p-1) \sum_{y_\beta\in \Grid} |\eta_\beta|^{p-2} \omega_\beta \right)\\
&+\tau (p-1) \left(\Lambda_{u_0}(|x_\alpha-x_\gamma|) + t_j \Lambda_f(|x_\alpha-x_\gamma|)\right)  \sum_{y_\beta\in \Grid} |\eta_\beta|^{p-2}  \omega_\beta  \\
&+ \tau \Lambda_f (|x_\alpha-x_\gamma|)\\
=&   \Lambda_{u_0}(|x_\alpha-x_\gamma|) + (t_j+\tau) \Lambda_f(|x_\alpha-x_\gamma|),
\end{split}
\end{equation*}
which concludes the proof. 
 \end{proof}
 
 We are now ready to state and prove the stability result: solutions with bounded data remain bounded (uniformly in the discretization parameters) for all times.\\
 \begin{proposition}\label{prop:stab}
 Under the assumptions of Proposition \ref{prop:regusch}, we have that
 \[
 \sup_{y_\alpha\in \Grid} |U^j_\alpha|\leq \|u_0\|_{L^\infty(\R^d)} + t_j  \|f\|_{L^\infty(\R^d)}, \quad \textup{for all} \quad j=0,\ldots, N.
 \]
 \end{proposition}
 \begin{proof}
 By assumption \eqref{as:u0f}, we have that
 \[
\sup_{y_\alpha\in \Grid}|U^0_\alpha|\leq \sup_{y_\alpha\in \Grid}|u_0(x_\alpha)| \leq \|u_0\|_{L^\infty(\R^d)}.
\]
Assume by induction that
\[
\sup_{y_\alpha\in \Grid} |U^j_\alpha|\leq \|u_0\|_{L^\infty(\R^d)} + t_j  \|f\|_{L^\infty(\R^d)}.
\]
Direct computations lead to
\[
\begin{split}
U^{j+1}_\alpha&=U^{j}_\alpha + \tau \sum_{y_\beta \in \Grid} |U^{j}_{\alpha+\beta}-U^{j}_\alpha|^{p-2}(U^{j}_{\alpha+\beta}-U^{j}_\alpha)\omega_\beta+ \tau f_\alpha\\
&=U^{j}_\alpha\left(1-\tau \sum_{y_\beta \in \Grid} |U^{j}_{\alpha+\beta}-U^{j}_\alpha|^{p-2} \omega_\beta \right)  + \tau \sum_{y_\beta \in \Grid} |U^{j}_{\alpha+\beta}-U^{j}_\alpha|^{p-2}U^{j}_{\alpha+\beta}\omega_\beta+ \tau f_\alpha.
\end{split}
\]
By Proposition \ref{prop:regusch}  we have that
\[
 |U^{j}_{\alpha+\beta}-U^{j}_\alpha|^{p-2}\leq (\Lambda_{u_0}(|y_\beta|) + t_j \Lambda_f(|y_\beta|))^{p-2},
\]
which together with assumptions \eqref{as:disc} and \eqref{as:CFL}  imply that
\[
\tau \sum_{y_\beta \in \Grid} |U^{j}_{\alpha+\beta}-U^{j}_\alpha|^{p-2} \omega_\beta \leq \tau  (\Lambda_{u_0}(r) + t_j \Lambda_f(r))^{p-2}\sum_{y_\beta \in \Grid} \omega_\beta \leq \frac{1}{p-1}\leq 1.
\]
Direct computations plus the induction hypothesis allow us to conclude that
\[
\begin{split}
|U^{j+1}_\alpha|\leq&\sup_{y_\alpha\in \Grid} |U^{j}_\alpha|\left(1-\tau \sum_{y_\beta \in \Grid} |U^{j}_{\alpha+\beta}-U^{j}_\alpha|^{p-2} \omega_\beta \right) \\
 &+ \tau \sup_{y_\alpha\in \Grid} |U^{j}_\alpha| \sum_{y_\beta \in \Grid} |U^{j}_{\alpha+\beta}-U^{j}_\alpha|^{p-2}\omega_\beta+ \tau \|f\|_{L^\infty(\R^d)}\\
 =&\sup_{y_\alpha\in \Grid} |U^{j}_\alpha| + \tau \|f\|_{L^\infty(\R^d)}\\
 =& \|u_0\|_{L^\infty(\R^d)} + (t_j +\tau)  \|f\|_{L^\infty(\R^d)},
\end{split}
\]
which concludes the proof.
 \end{proof}

 \subsection{Time equicontinuity for a discrete in time scheme} 
 
 Now we extend the scheme from $\Grid$ to $\R^d$ by considering $U:\R^d\times \mathcal{T}_{\tau}$ defined by
\begin{equation}\label{eq:numsch_ext}
\begin{cases}
U^j(x)= U^{j-1}(x) + \tau \left(D_p^h U^{j-1}(x)+f(x)\right), & x \in \R^d,\,  j=1,\ldots,N,\\
U^0(x)=u_0(x) & x \in \R^d.
\end{cases}
\end{equation}
\begin{remark}
Clearly, if we restrict the solution of \eqref{eq:numsch_ext} to $\Grid$, we recover the solution of \eqref{eq:numsch}.
\end{remark}

\begin{proposition}[Continuous dependence on the data]\label{prop:contdep}
Assume \eqref{as:u0f}, \eqref{as:disc}, $p\geq2$,   $r, h,\tau >0$ and \eqref{as:CFL}. Let $U,\widetilde{U}$ be the solutions of \eqref{eq:numsch} corresponding to $u_0, \widetilde{u}_0$ and $f,\widetilde{f}$. For every $j=0,\ldots,N$, we have
 \[
\|U^j-\widetilde{U}^j\|_{L^\infty(\R^d)} \leq  \|u_0- \widetilde{u}_0\|_{L^\infty(\R^d)}  + t_j  \|f- \widetilde{f}\|_{L^\infty(\R^d)}. \]
\end{proposition}
\begin{proof}
By assumption \eqref{as:u0f}, we have that
\[
\|U^0-\widetilde{U}^0\|_{L^\infty(\R^d)}= \|u_0-\widetilde{u}_0\|_{L^\infty(\R^d)}.
\]
Assume by induction that
\[
\|U^j-\widetilde{U}^j\|_{L^\infty(\R^d)}= \|u_0-\widetilde{u}_0\|_{L^\infty(\R^d)} + t_j \|f-\widetilde{f}\|_{L^\infty(\R^d)}.
\]
Similar computations as the ones in the proof of Proposition \ref{prop:regusch} yield
\begin{equation}\label{eq:pres2}
\begin{split}
U^{j+1}(x)-&\widetilde{U}^{j+1}(x)= (U^{j}(x)-\widetilde{U}^{j}(x))\left(1-\tau (p-1) \sum_{y_\beta\in \Grid} |\eta_\beta|^{p-2} \omega_\beta \right)\\
&+\tau (p-1) \sum_{y_\beta\in \Grid} |\eta_\beta|^{p-2} (U^{j}(x+y_\beta)-\widetilde{U}^{j}(x+y_\beta)) \omega_\beta  + \tau  (f(x)-\widetilde{f}(x)),
\end{split}
\end{equation}
where $\eta_\beta\in \R$ is some number between $(U^{j}(x+y_\beta)-U^{j}(x))$ and $(\widetilde{U}^{j}(x+y_\beta)-\widetilde{U}^{j}(x))$. From here, the proof follows as in the proof of Proposition \ref{prop:regusch}. 
\end{proof}

\begin{proposition}[Equicontinuity in time]\label{prop:equitime}
Assume \eqref{as:u0f}, \eqref{as:disc}, $p\geq2$,   $r,h,\tau >0$ and \eqref{as:CFL}. Let $U$ be the solution of \eqref{eq:numsch_ext}. Then
\[
\begin{split}
\|U^{j+k}-U^{j}\|_{L^\infty(\R^d)} &\leq \widetilde{K} (t_k)^{\frac{a}{2+(1-a)(p-2)}} +  \|f\|_{L^\infty(\R^d)}t_k=:\overline{\Lambda}_{u_0,f}(t_k),
\end{split}
\]
with 
\begin{equation}\label{eq:ctefam}
\tilde{K} = 4^{\frac{1+(1-a)(p-1)}{2+(1-a)(p-2)}} L_{u_0}^{\frac{p}{2+(1-a)(p-2)}}((p-1)  K_1^{p-2}K_2 M)^\frac{a}{2+(1-a)(p-2)},
\end{equation}
where $M$ comes from assumption \eqref{as:disc}, and $K_1$ and $K_2$ are constants given in Section \ref{sec:ctes} (depending on a certain choice of mollifiers).
\end{proposition}
\begin{remark}
 Actually, a close inspection of the proof reveals that for $u_0\in C^2_b(\R^d)$ we can get
\[
\|U^{j+k}-U^{j}\|_{L^\infty(\R^d)} \lesssim t_k.
\]
\end{remark}
\begin{proof}
Consider a mollification of the initial data $u_{0,\delta}=u_0* \rho_\delta$ where  $\rho_\delta(x)$ is a standard mollifier (as defined in  Appendix  \ref{sec:moll}). Let $(U_\delta)^j$ be the corresponding solution of \eqref{eq:numsch_ext} with $u_{0,\delta}$ as initial data. Then,
\[
\|(U_\delta)^1-(U_\delta)^0\|_{L^\infty(\R^d)} \leq \tau\|D_p^hu_{0,\delta}\|_{L^\infty(\R^d)}  + \tau \|f\|_{L^\infty(\R^d)}.
\]
Define $\widetilde{U}^j_\delta:= U^{j+1}_\delta$ for all $j=0,\ldots,N$. Clearly, $\widetilde{U}^j_\delta$ is the unique solution of \eqref{eq:numsch_ext} with initial data $\widetilde{U}^0_\delta=U^{1}_\delta$ and right hand side  $f$. By Proposition \ref{prop:contdep} 
\[
\begin{split}
\|U^{j+1}_\delta-U^j_\delta\|_{L^\infty(\R^d)} &=\|\widetilde{U}^{j}_\delta-U^j_\delta\|_{L^\infty(\R^d)} \leq \|\widetilde{U}^{0}_\delta-U^0_\delta\|_{L^\infty(\R^d)}=\|U^1_\delta-U^0_\delta\|_{L^\infty(\R^d)} \\
&\leq \tau\|D_p^hu_{0,\delta}\|_{L^\infty(\R^d)} + \tau \|f\|_{L^\infty(\R^d)}.
\end{split}
\]
A repeated use of the triangle inequality yields
\begin{equation}\label{eq:triangle}
\begin{split}
\|U^{j+k}_\delta-U^j_\delta\|_{L^\infty(\R^d)}&\leq \sum_{i=0}^{k-1} \|U^{j+i+1}_\delta-U^{j+i}_\delta\|_{L^\infty(\R^d)}\leq(k\tau) \|D_p^hu_{0,\delta}\|_{L^\infty(\R^d)} + (k\tau) \|f\|_{L^\infty(\R^d)}.
\end{split}
\end{equation}
The symmetry of the weights $\omega_\beta$ together with  Lemma \ref{lem:pineq1}  implies 
\begin{equation}\label{eq:Dp}
\begin{split}
|D_p^hu_{0,\delta}(x)|&=\frac{1}{2} \left|\sum_{y_\beta\in \Grid} \left(J_p(u_{0,\delta}(x+y_\beta)-u_{0,\delta}(x))-J_p(u_{0,\delta}(x)-u_{0,\delta}(x-y_\beta))\right) \omega_{\beta}\right|\\
&\leq \frac{p-1}{2} \sum_{y_\beta\in \Grid} 
\max\{|u_{0,\delta}(x+y_\beta)-u_{0,\delta}(x)|,|u_{0,\delta}(x)-u_{0,\delta}(x-y_\beta)|\}^{p-2} \times\\
&
\hspace{5cm}\times\left|u_{0,\delta}(x+y_\beta)+u_{0,\delta}(x-y_\beta)-2u_{0,\delta}(x) \right|\omega_\beta.
\end{split}
\end{equation}
Now note that, by the $a$-H\"older regularity of $u_0$ given by assumption \eqref{as:u0f}, Lemma \ref{lem:a1} and Lemma \ref{lem:a2} imply
\begin{equation}\label{eq:mollest}
|u_{0,\delta}(x\pm y_\beta)-u_{0,\delta}(x)|\leq K_1L_{u_0}\delta^{a-1}|y_\beta|, \quad |u_{0,\delta}(x+y_\beta)+u_{0,\delta}(x-y_\beta)-2u_{0,\delta}(x) |\leq K_2 L_{u_0} \delta^{a-2}|y_\beta|^2,
\end{equation}
where $K_1$ and $K_2$ depend only on the mollifier $\rho$.  Now note that, by \eqref{as:disc}, we have
\begin{equation}\label{eq:weightest}
\sum_{y_\beta\in \Grid} |y_\beta|^p\omega_\beta \leq M.
\end{equation}
Combining \eqref{eq:triangle} and \eqref{eq:weightest}, we obtain
\[
\begin{split}
\|U^{j+k}_\delta-U^j_\delta\|_{L^\infty(\R^d)}&\leq \frac{p-1}{2}  t_k (K_1L_{u_0}\delta^{a-1})^{p-2} K_2 L_{u_0} \delta^{a-2}\sum_{y_\beta\in \Grid} |y_\beta|^p\omega_\beta+ t_k  \|f\|_{L^\infty(\R^d)}\\
&\leq \widehat{K} \delta^{(a-1)(p-2)+(a-2)}t_k+  \|f\|_{L^\infty(\R^d)}t_k,
\end{split}
\]
with $\widehat{K}=\frac{p-1}{2}  K_1^{p-2}K_2 L_{u_0}^{p-1}M$. Using the triangle inequality,  the above estimate  and applying Proposition \ref{prop:contdep}  several times we obtain
\[
\begin{split}
\|U^{j+k}-U^{j}\|_{L^\infty(\R^d)}&\leq \|U^{j+k}-U^{j+k}_\delta\|_{L^\infty(\R^d)}+\|U^{j+k}_\delta-U^{j}_\delta\|_{L^\infty(\R^d)}+\|U^{j}-U^{j}_\delta\|_{L^\infty(\R^d)}\\
&\leq 2\|u_0-u_{0,\delta}\|_{L^\infty(\R^d)}+\widehat{K} \delta^{(a-1)(p-2)+(a-2)}t_k+  \|f\|_{L^\infty(\R^d)}t_k\\
&\leq 2 L_{u_0}\delta^a +\widehat{K} \delta^{(a-1)(p-2)+(a-2)}t_k+   \|f\|_{L^\infty(\R^d)}t_k.
\end{split}
\]
 By choosing  $\delta=(\frac{\widehat{K}}{2L_{u_0}}t_k)^{\frac{1}{2+(1-a)(p-2)}}$ in the above estimate, we get the desired result
\[
\|U^{j+k}-U^{j}\|_{L^\infty(\R^d)} \leq \tilde{K} (t_k)^{\frac{a}{2+(1-a)(p-2)}} +   \|f\|_{L^\infty(\R^d)}t_k,
\]
with 
\[
\begin{split}
\tilde{K}&=4L_{u_0} \left(\frac{\widehat{K}}{2L_{u_0}} \right)^\frac{a}{2+(1-a)(p-2)}\\
&=4^{1-\frac{a}{2+(1-a)(p-2)}}L_{u_0}((p-1)  K_1^{p-2}K_2 L_{u_0}^{p-2}M)^\frac{a}{2+(1-a)(p-2)}\\
&= 4^{\frac{2+(1-a)(p-2) -a}{2+(1-a)(p-2)}} L_{u_0}^{\frac{p}{2+(1-a)(p-2)}}((p-1)  K_1^{p-2}K_2 M)^\frac{a}{2+(1-a)(p-2)}.
\end{split}
\]
\end{proof}
\subsection{Equiboundedness and equicontinuity estimates for a scheme in $\R^d\times[0,T]$}

We now need to extend the numerical scheme in time in a continuous way. This is done by continuous interpolation, i.e., 
\begin{equation}\label{eq:interp}
U(x,t):= \frac{t_{j+1}-t}{\tau} U^j(x)+ \frac{t-t_j}{\tau}U^{j+1}(x) \quad \textup{if} \quad t\in [t_j,t_{j+1}] \quad \textup{for some} \quad j=0,\ldots,N,
\end{equation}
where $U^j$ is the solution of \eqref{eq:numsch_ext}. 
\begin{remark}\label{rem:schemenotimegrid}
It is standard to check that, for all $t\in[t_j,t_{j+1}]$, we have that the original scheme is preserved also outside the grid points, i.e., 
\begin{equation}
U(x,t)=U(x,t_j)+(t-t_j) D_p^h U(x,t_j)+ (t-t_j) f(x).
\end{equation}
\end{remark}

We have the following result.
\begin{proposition}[Stability and equicontinuity]\label{prop:stabcont}
Assume \eqref{as:u0f}, \eqref{as:disc}, $p\geq2$,   $r,h,\tau>0$ and \eqref{as:CFL}. Let $U$ be the solution of \eqref{eq:interp}. Then
\begin{enumerate}[(a)]
\item \emph{(Equiboundedness)} $\|U\|_{L^\infty(\R^N\times[0,T])} \leq  \|u_0\|_{L^\infty(\R^d)} +T \|f\|_{L^\infty(\R^d)} $,
\item \emph{(Equicontinuity)} For any $x,z\in \R^d$ and $t,\tilde{t}\in [0,T]$ we have that
\[
|U(x,t)-U(z,\tilde{t})| \leq \Lambda_{u_0}(|x-z|) +T \Lambda_{f}(|x-z|) + 3 \overline{\Lambda}_{u_0,f}(|\tilde{t}-t|).
\]
\end{enumerate}
\end{proposition}
\begin{proof}
Equiboundedness follows easily from a continuous in space version of Proposition \ref{prop:stab}, since
\[
\begin{split}
|U(x,t)|&\leq \frac{t_{j+1}-t}{\tau} \sup_{x\in \R^d} |U^j(x)|+ \frac{t-t_j}{\tau} \sup_{x\in \R^d}|U^{j+1}(x)|\\
& \leq \frac{t_{j+1}-t}{\tau} \left(  \|u_0\|_{L^\infty(\R^d)} + T  \|f\|_{L^\infty(\R^d)}\right)+ \frac{t-t_j}{\tau} \left(  \|u_0\|_{L^\infty(\R^d)} + T  \|f\|_{L^\infty(\R^d)}\right)\\
&\leq \|u_0\|_{L^\infty(\R^d)} +T \|f\|_{L^\infty(\R^d)}.
\end{split}
\]
Equicontinuity in space follows from the translation invariance of the scheme and Proposition \ref{prop:contdep}: 
\[
|U(x+y,t)-U(x,t)|\leq \|u_0(\cdot+y)-u_0\|_{L^\infty(\R^d)} +T \|f(\cdot+y)-f\|_{L^\infty(\R^d)}.
\]
To prove equicontinuity in time, we first consider $t,\tilde{t} \in [t_j,t_{j+1}]$ for some $j=0,\ldots,N-1$. In this case we have
\[
\begin{split}
U(x,t)-U(x,\tilde{t})&=\left(\frac{t_{j+1}-t}{\tau} U^j(x)+ \frac{t-t_j}{\tau}U^{j+1}(x)\right)-\left(\frac{t_{j+1}-\tilde{t}}{\tau} U^j(x)+ \frac{\tilde{t}-t_j}{\tau}U^{j+1}(x)\right)\\
&= \frac{t-\tilde{t}}{\tau} \left(U^{j+1}(x)- U^j(x)\right).
\end{split}
\]
Then, from Proposition \ref{prop:equitime}, we get
\[
|U(x,t)-U(x,\tilde{t})|\leq  |t-\tilde{t}| \frac{ \overline{\Lambda}_{u_0,f}(\tau) }{\tau}
\]
Note that the function $g(\tau)=\frac{ \overline{\Lambda}_{u_0,f}(\tau)}{\tau}$
is decreasing. Thus, since $|t-\tilde{t}|\leq \tau$, we have  $g(\tau)\leq  g(|t-\tilde{t}|)$. It follows that
\[
\begin{split}
|U(x,t)-U(x,\tilde{t})|&\leq \overline{\Lambda}_{u_0,f}(|t-\tilde{t}|). 
\end{split}
\]
Now consider $t\in [t_j,t_{j+1})$ and $\tilde{t}\in [t_{j+k},t_{j+k+1})$ for $k\geq1$. By the triangle inequality, the previous step and Proposition \ref{prop:equitime} 
\[
\begin{split}
|U(x,t)-U(x,\tilde{t})|&\leq |U(x,t)-U(x,t_{j+1})|+|U(x,t_{j+k})-U(x,\tilde{t})|+|U(x,t_{j+1})-U(x,t_{j+k})|\\
&\leq \overline{\Lambda}_{u_0,f}(|t_{j+1}-t|) + \overline{\Lambda}_{u_0,f}(|\tilde{t}-t_{j+k}|) +\overline{\Lambda}_{u_0,f}(|t_{j+k}-t_{j+1}|) .
\end{split}
\]
Since $t\leq t_{j+1} \leq \tilde{t}$ and $t\leq t_{j+k} \leq \tilde{t}$, the above estimate yields
\[
|U(x,t)-U(x,\tilde{t})|\leq 3 \overline{\Lambda}_{u_0,f}(|\tilde{t}-t|) .
\]
Finally, we conclude space-time equicontinuity combining the above estimates to get
\[
\begin{split}
|U(x,t)-U(z,\tilde{t})|&\leq |U(x,t)-U(z,t)|+|U(z,t)-U(z,\tilde{t})|\\
&\leq \Lambda_{u_0}(|x-z|) +T \Lambda_{f}(|x-z|) + 3\overline{\Lambda}_{u_0,f}(|\tilde{t}-t|).\qedhere
\end{split}
\]
\end{proof}
By Arzel\`a-Ascoli, we obtain as a corollary that, up to a subsequence, the numerical solution converges  locally uniformly to a limit. 
\begin{corollary}\label{coro:compactness}
Assume the hypotheses of Proposition \ref{prop:stabcont}. Let $\{U_h\}_{h>0}$ be a sequence of solutions of \eqref{eq:interp}. Then, there exist a subsequence $\{U_{h_l}\}_{l=1}^\infty$ and a function $v\in C_b(\R^d\times[0,T])$ such that
\[
U_{h_l}\to v \quad \textup{as} \quad l\to\infty \quad \textup{locally uniformly in $\R^N\times[0,T]$}.
 \]
\end{corollary}

\section{Convergence of the numerical scheme}

From Corollary \ref{coro:compactness}, we have that the sequence of numerical solutions has a subsequence converging locally uniformly to some function $v$. We will now show that $v$ is a viscosity solution of \eqref{eq:ParabProb}.

\begin{theorem}\label{thm:convergence}
Let the assumptions of Corollary \ref{coro:compactness} hold. Then $v$ is a viscosity solution of \eqref{eq:ParabProb}.
\end{theorem}
\begin{proof} For notational simplicity, we avoid the subindex $j$ and consider 
\[
U_{h}\to v \quad \textup{as} \quad h\to0 \quad \textup{locally uniformly in $\R^N\times[0,T]$}.
\]
First of all, by the local uniform convergence,
\[
v(x,0)=\lim_{h\to0} U_h(x,0)=u_0(x),
\]
locally uniformly. We will now show that $v$ is a viscosity supersolution. The proof that $v$ is a viscosity subsolution is similar. 

Now let $\varphi$ be a suitable test function for $v$ at $(x^*,t^*)\in \R^d\times (0,T)$.  We may assume that $\varphi$ satisfies
\begin{enumerate}[(i)]
\item $\varphi(x^*,t^*)=u(x^*,t^* )$,
\item $u(x,t)>\displaystyle \varphi(x,t)$ for all $(x,t)\in B_R(x^*)\times(t^*-R,t^*]\setminus \left(x^*,t^*\right)$.
\end{enumerate}
The local uniform convergence ensures (see Section 10.1.1 in \cite{Eva98}) that there exists a sequence $\{(x^h,t^h)\}_{h>0}$ such that
\begin{enumerate}[(i)]
\item $\varphi(x^h,t^h)-U_h(x^h,t^h)=\sup_{(x,t)\in B_R(x^h)\times(t^h-R,t^h]}\{\varphi(x,t)-U_h(x,t)\}=:M_h $,
\item  $\varphi(x^h,t^h)-U_h(x^h,t^h)\geq \displaystyle \varphi(x,t)-U_h(x,t)$ for all $(x,t)\in B_R(x^h)\times(t^h-R,t^h]\setminus(x^h,t^h)$
\end{enumerate}
and
\[
(x^h,t^h)\to (x^*,t^*) \quad \textup{as} \quad h\to0.
\]
Now consider $t_j\in \mathcal{T}_{\tau}$ such that $t^h\in[t_j,t_{j+1}]$ (note that the index $j$ might depend on $h$, but this fact plays no role in the proof). By Remark \ref{rem:schemenotimegrid}, 
\[
U_h(x^h,t^h)= U_h(x^h,t_j)+ (t^h-t_j) \sum_{y_\beta \in \Grid} J_p(U_h(x^h+y_\beta,t_j)-U_h(x^h,t_j))\omega_\beta + (t^h-t_j) f(x^h).
\]
Define $\widetilde{U}_h=U_h+M_h$. It is clear that
\[
\widetilde{U}_h(x^h,t^h)= \widetilde{U}_h(x^h,t_j)+ (t^h-t_j) \sum_{y_\beta \in \Grid} J_p(\widetilde{U}_h(x^h+y_\beta,t_j)-\widetilde{U}_h(x^h,t_j))\omega_\beta+(t^h-t_j) f(x^h).
\]
Clearly,  $\widetilde{U}_h(x^h,t^h)=\varphi(x^h,t^h)$ and $\widetilde{U}_h\geq\varphi$, which implies that
\begin{equation}\label{eq:aux1}
\varphi(x^h,t^h)= \widetilde{U}_h(x^h,t_j)+ (t^h-t_j) \sum_{y_\beta \in \Grid} J_p(\widetilde{U}_h(x^h+y_\beta,t_j)-\widetilde{U}_h(x^h,t_j))\omega_\beta+(t^h-t_j) f(x^h).
\end{equation}
Now consider the function $g:\R\to\R$ given by
\[
g(\xi)=\xi+(t^h-t_j) \sum_{y_\beta \in \Grid} J_p(\widetilde{U}(x^h+y_\beta,t_j)-\xi)\omega_\beta
\]
and note that
\[
g'(\xi)=1-(t^h-t_j) (p-1) \sum_{y_\beta \in \Grid} |\widetilde{U}(x^h+y_\beta,t_j)-\xi|^{p-2}\omega_\beta.
\]
We will check now that $g'(\xi)\geq0$ for any $\xi \in [\varphi(x^h,t_j), \widetilde{U}(x^h,t_j)]$. Indeed, 
\[
\begin{split}
 |\widetilde{U}(&x^h+y_\beta,t_j)-\xi| \leq  |\widetilde{U}(x^h+y_\beta,t_j)-\widetilde{U}(x^h,t_j)|+|\widetilde{U}(x^h,t_j)-\xi|\\
 &\leq  |U(x^h+y_\beta,t_j)-U(x^h,t_j)|+|\widetilde{U}(x^h,t_j)-\varphi(x^h,t_j)|\\
&\leq  |U(x^h+y_\beta,t_j)-U(x^h,t_j)|+|U(x^h,t_j)-U(x^h,t^h)|+|\varphi(x^h,t^h)-\varphi(x^h,t_j)|\\
&\leq \Lambda_{u_0}(|y_\beta|) +T \Lambda_{f}(|y_\beta|) + 3 \overline{\Lambda}_{u_0,f}(|t^h-t_j|)+|t^h-t_j| \|\partial_t \varphi\|_{L^\infty(B_R(x^h)\times[t^h-R, t^h])}\\
&\leq  \Lambda_{u_0}(|y_\beta|) +T \Lambda_{f}(|y_\beta|) +  3 \overline{\Lambda}_{u_0,f}(\tau)+\tau \|\partial_t \varphi\|_{L^\infty(B_R(x^h)\times[t^h-R, t^h])},
 \end{split}
\]
where we have used that $\widetilde{U}(x^h,t^h)=\varphi(x^h,t^h)$, Proposition \ref{prop:stabcont} and the fact that $|t^h-t_j|\leq \tau$.  By \eqref{as:CFL}, and taking $\tau$ small enough, we have 
\[
\begin{split}
3 \overline{\Lambda}_{u_0,f}(\tau)&+\tau \|\partial_t \varphi\|_{L^\infty(B_R(x^h)\times[t^h-R, t^h])}\\
& \leq 3\widetilde{K} \tau^{\frac{a}{2+(1-a)(p-2)}} + \left(3\|f\|_{L^\infty(\R^d)} + \|\partial_t \varphi\|_{L^\infty(B_R(x^h)\times[t^h-R, t^h])}\right)\tau\\
& \leq (3\widetilde{K}+1)\tau^{\frac{a}{2+(1-a)(p-2)}}\\
& \leq (3\widetilde{K}+1) r^a.
\end{split}
\]
Thus,
\[ 
\begin{split}
g'(\xi) &\geq 1-(t^h-t_j) (p-1) \sum_{y_\beta \in \Grid} |  \Lambda_{u_0}(|y_\beta|) +T \Lambda_{f}(|y_\beta|) + (3\widetilde{K}+1) r^a|^{p-2}\omega_\beta\\
&\geq 1-\tau(p-1)  (  L_{u_0}+T L_f + 3\widetilde{K}+1)^{p-2}r^{a(p-2)}\sum_{y_\beta \in \Grid}\omega_\beta \\
&\geq 1-\tau \frac{M(p-1)  (  L_{u_0} +T L_{f} + 3\widetilde{K}+1)^{p-2}}{ r^{2+(1-a)(p-2)}}\\
&\geq 0,
\end{split}
\]
where we have used \eqref{as:disc} and where the last inequality is due to the \eqref{as:CFL} condition.
We can use this fact in \eqref{eq:aux1} to get
\[
\begin{split}
\varphi(x^h,t^h)&= \widetilde{U}_h(x^h,t_j)+ (t^h-t_j) \sum_{y_\beta \in \Grid} J_p(\widetilde{U}_h(x^h+y_\beta,t_j)-\widetilde{U}_h(x^h,t_j))\omega_\beta+(t^h-t_j)  f(x^h)\\
&\geq \varphi(x^h,t_j)+ (t^h-t_j)  \sum_{y_\beta \in \Grid} J_p(\widetilde{U}_h(x^h+y_\beta,t_j)-\varphi(x^h,t_j))\omega_\beta+(t^h-t_j)  f(x^h)\\
&\geq \varphi(x^h,t_j)+ (t^h-t_j) \sum_{y_\beta \in \Grid} J_p(\varphi(x^h+y_\beta,t_j)-\varphi(x^h,t_j))\omega_\beta+(t^h-t_j)  f(x^h).
\end{split}
\]
Consistency \eqref{as:cons} yields 
\[
\partial_t\varphi(x^h,t^h) + o(\tau)\geq \plap \varphi(x^h,t_j)+ o_h(1)+f(x^h).
\]
Passing to the limit as $h,\tau\to0$, we get the desired result by the regularity of $\varphi$ and the fact that $t^h,t_j\to t^*$  and $x^h\to x^*$ as $h\to0$.
\end{proof}
We are now ready to prove convergence of the scheme.
\begin{proof}[Proof of Theorem \ref{theo:main}]
By Corollary \ref{coro:compactness} and Theorem \ref{thm:convergence}, we know that, up to a subsequence, the sequence $U_h$ converges to a viscosity solution of \eqref{eq:ParabProb}. Moreover, since viscosity solutions are unique (cf. Theorem \ref{teo:main2}), the whole sequence converges to the same limit.
\end{proof}

\section{Discretizations}\label{sec:discretizations}

In this section, we present two examples of discretizations and verify that the assumptions \eqref{as:cons} and \eqref{as:disc} are satisfied. Moreover, we also give the precise form of corresponding CFL-condition.

\subsection{Discretization in dimension $d=1$} We consider the following finite difference discretization of $\plap$ in dimension $d=1$
\[
D_p^h \phi(x) = \frac{J_p(\phi(x+h)-\phi(x))+J_p(\phi(x-h)-\phi(x))}{h^p}.
\]
A proof of consistency \eqref{as:cons} can be found in Theorem 2.1 in \cite{dTLi20}. Assumption \eqref{as:disc} is trivially true for $r=h$ since
\[
\omega_1=\omega_{-1}= \frac{1}{h^p} \quad \textup{and} \quad \omega_\beta=0 \quad \textup{otherwise}.
\]
so that
\[
\sum_{y_\beta\in \Grid} \omega_\beta=\frac{2}{h^p}
\]
\subsection{Discretization in dimension $d>1$}\label{subsec:disc2}
The following discretization was introduced in \cite{dTLi22}:
\[
D_p^h\phi(x)= \frac{h^d}{\mathcal{D}_{d,p}\, \omega_d\, r^{p+d}} \sum_{y_\beta\in B_r} J_p(\phi(x+y_\beta)-\phi(x)),
\]
where $\omega_d$ denotes the measure of the unit ball in $\R^d$, the relation between $r$ and $h$ is given by 
\begin{equation}\label{eq:hr}
h=\begin{cases}
o(r^\frac{p}{p-1}), & \quad \textup{if} \quad p \in (2,3],\\
o(r^{\frac{3}{2}}),&  \quad \textup{if} \quad p\in (3,\infty),\\
\end{cases}
\end{equation}
and  $\mathcal{D}_{d,p}= \frac{d}{2(d+p)}\fint_{\partial B_1}|y_1|^p\dd \sigma (y)$. When $p\in \N$, a more explicit value of this constant is given in \cite{dTLi22}. In general, the explicit value is given by 
\[
\mathcal{D}_{d,p}
	=
	\frac{d}{4\sqrt\pi}\cdot\frac{p-1}{d+p}\cdot\frac{\Gamma(\frac{d}{2})\Gamma(\frac{p-1}{2})}{\Gamma(\frac{d+p}{2})}.
\]
A proof of consistency \eqref{as:cons} can be found in Theorem 1.1 in \cite{dTLi22}. Assumption \eqref{as:disc} trivially holds for $h=o(r^\alpha)$ for some $\alpha>0$ according to \eqref{eq:hr} since
\[
\omega_\beta=\omega_{-\beta }= \frac{h^d}{\mathcal{D}_{d,p}\, \omega_d\, r^{p+d}} \quad \textup{if} \quad  |h\beta|<r \quad \textup{and} \quad \omega_\beta=0 \quad \textup{otherwise}.
\]
To check \eqref{as:disc} we rely on the following estimate given in the proof of Theorem 1.1 in \cite{dTLi22}:
\[
\sum_{y_\beta\in B_r}h^d \leq |B_{r+\sqrt{d}h}|.
\]
In particular,  taking for example $h\leq r/\sqrt{d}$, we have
\[
\sum_{y_\beta\in B_r} \omega_\beta= \frac{1}{\mathcal{D}_{d,p}r^{p}}\frac{|B_{r+\sqrt{d}h}|}{|B_{r}|}\leq \frac{2^d}{\mathcal{D}_{d,p}r^{p}}.
\]
\section{Numerical experiments}
We will perform the numerical tests comparing the numerical solution with the explicit Barenblatt solution of \eqref{eq:ParabProb}. For $p>2$ this is given by
\[
B(x,t)=K t^{-\alpha}\left(1- \left(\frac{|x|}{t^\beta}\right)^{\frac{p}{p-1}}\right)^{\frac{p-1}{p-2}}_+, 
\]
where the constants are,
\[
\alpha=\frac{d}{d(p-2)+p}, \quad \beta=\frac{1}{d(p-2)+p}, \quad \textup{and} \quad K=\left(\frac{p-2}{p}\beta^{\frac{1}{p-1}}\right)^{{\frac{p-1}{p-2}}}.
\]
\subsection{Simulations in dimension $d=1$}
We consider the initial condition
\[
u_0(x)=K\left(1-|x|^{\frac{p}{p-1}}\right)^{\frac{p-1}{p-2}}
\]
and $f=0$. The corresponding solution of problem \eqref{eq:ParabProb} is given by (see \cite{KaVa88}) 
\[
B(x,t)=K (t+1)^{-\alpha}\left(1- \left(\frac{|x|}{(t+1)^\beta}\right)^{\frac{p}{p-1}}\right)^{\frac{p-1}{p-2}}_+ .
\] 
Let us now comment on the CFL-condition \eqref{as:CFL}.  Clearly, $u_0$ is a Lipschitz function, and we can estimate its Lipschitz constant as follows	
\[
L_{u_0}=\sup_{x\in [-1,1]} \left|\frac{du_0}{dx}(x)\right|=\sup_{r\in[0,1]}\left\{K \frac{p}{p-2}\left(1-r^{\frac{p}{p-1}}\right)^{\frac{1}{p-2}}r^{\frac{1}{p-1}}\right\}\leq K \frac{p}{p-2}.
\]
Thus, for all $p>2$, the CFL condition \eqref{as:CFL} can be take as $\tau\sim h^2$ (since $f=0$ in this case). 
 For completeness, we find the value of $K$ in dimension $d=1$.  Note that
\[
K=\left(\frac{p-2}{p}\frac{1}{(2(p-1))^{\frac{1}{p-1}}}\right)^{{\frac{p-1}{p-2}}}=\left(\frac{p-2}{p}\right)^{{\frac{p-1}{p-2}}}\frac{1}{(2(p-1))^{\frac{1}{p-2}}},
\]
so that
$
L_{u_0}\leq \left(\frac{p-2}{2p(p-1)}\right)^{{\frac{1}{p-2}}}.
$

In Figure \ref{fig:errors222}, we show the numerical errors obtained. As it can be seen there, the errors seem to behave like  $O(h^{p/(p-1)})$.

\begin{figure}[h!]
\includegraphics[width=0.9\textwidth]{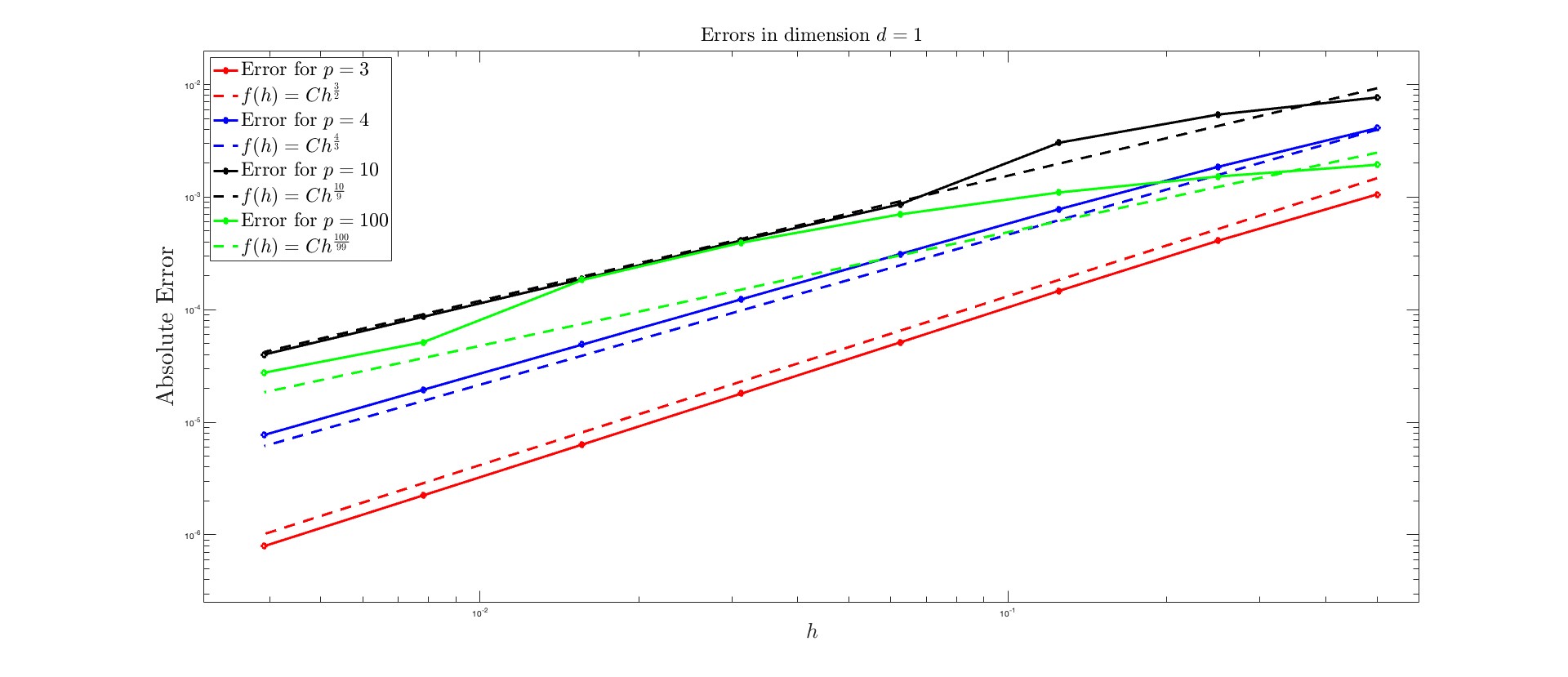}
\caption{Errors in dimension $d=1$ for $p=3,4,10,100$. }
\label{fig:errors222}
\end{figure}

\appendix
\section{Estimates for mollified H\"older continuous functions}\label{sec:moll}
Here we present some explicit estimates for mollifications needed in the proof of  equicontinuity in Lemma \ref{prop:equitime}.  Let $\tau:[0,\infty)\to \R$ be a smooth function such that $\supp\tau\subset [0,1 ]$. Define $\rho:\R^d\to \R$  given by  $\rho(x)=\frac{M}{\omega_d}\tau(|x|)$ where $\omega_d$ is the measure of the unit sphere in dimension $d$ and $M=M(d)$ is a constant defined by
\[
M=\left(\int_0^1 \tau (r) r^{d-1}\dd r\right)^{-1}.
\]
In this way, we have that $\int_{B_1}\rho(x)\,dx = 1$.
For $\delta>0$ define also
$$
\rho_\delta(x)=\frac{1}{\delta^d}\rho\left(\frac{x}{\delta}\right).
$$
Then, for a function $f\in L^1_{\text{loc}}$ we define the mollification of $f$ as
$$
f_\delta(x)=(f*\rho_\delta)(x)=\int_{B_\delta}\rho_\delta(y)f(x-y)\dd y =\int_{\R^n}\rho_\delta(x-y)f(y) \dd y.
$$
The lemma below gives an estimate of the Lipschitz seminorm of $f_\delta$ when $f$ is  an  $\alpha$-H\"older continuous function  for some $\alpha\in (0,1]$.
\begin{lemma} \label{lem:a1}  Let $\alpha\in(0,1]$.  Consider a function $f\in C^\alpha(\R^n)$ with $
|f(x)-f(y)|\leq L|x-y|^\alpha 
$ for all $x,y\in \R^d$. 
 Then, for all $x,y\in \R^d$, we have that
\[|f_\delta(x)-f_\delta(y)|\leq K_1 L|x-y|\delta^{\alpha-1}, \quad \textup{with} \quad K_1= M  \int_0^1|\tau'(r)| r^{d-1}\dd r.
\]
\end{lemma}
\begin{proof}
Since $  \int_{\R^n}\nabla \rho_\delta(y)\,dy=\int_{\R^n}\nabla \rho(y)\,dy= 0
$,
it follows that 
\begin{equation}
\begin{split}
|\nabla f_\delta (x)|&=\left|\int_{B_\delta}\nabla \rho_\delta(y)\left(f(x-y)-f(x)\right) \dd y\right|\leq L\int_{B_\delta}\left|\nabla \rho_\delta(y)\right||y|^\alpha \dd y\\
&\leq \frac{M}{\omega_d} L \delta^\alpha \int_{B_\delta}\frac{1}{\delta^{d+1}}\left|\tau'\left(\frac{|y|}{\delta}\right)\right| \dd y=M L \delta^{\alpha-1} \int_0^1|\tau'(r)| r^{d-1}\dd r.
\end{split}
\end{equation}
Thus,
\[
|f_\delta(x)-f_\delta(y)|\leq \|\nabla f_\delta\|_{L^\infty(\R^d)}|x-y|=K_1L|x-y|\delta^{\alpha-1}.\qedhere
\]
\end{proof}

The lemma below gives an estimate of the second order central  difference  quotients of $f_\delta$ when $f$ is  an  $\alpha$-H\"older continuous function  for some $\alpha \in(0,1]$. 
 \begin{lemma} \label{lem:a2}  Let $\alpha\in(0,1]$.  Consider a function $f\in C^\alpha(\R^n)$ with $
|f(x)-f(y)|\leq L|x-y|^\alpha 
$ for all $x,y\in \R^d$. 
 Then, for all $x,y\in \R^d$, we have that$$
|f_\delta(x+y)+f_\delta(x-y)-2f_\delta(x)|\leq K_2L|y|^2\delta^{\alpha-2},
$$
with
$$
 K_2=  M\int_0^1\left(\frac{|\tau'(r)|}{r} +|\tau''(r)|\right)r^{d-1}\dd r .
$$
\end{lemma}
\begin{proof}
 We note that we have the following formula for the second order derivatives of $\rho_\delta$: 
$$
\partial_{ij} \rho_\delta(y)=\left(\frac{\delta_{ij}}{|y|}-\frac{y_i y_j}{|y|^3}\right)\frac{1}{\delta^{d+1}}\tau'\left(\frac{|y|}{\delta}\right)	+\frac{y_i y_j}{\delta^{d+2}|y|^2}\tau''\left(\frac{|y|}{\delta}\right),
$$
so that
\[
\langle D^2 \rho_\delta (y) \xi,\xi\rangle \leq \left(\frac{1}{\delta^{d+1}|y|}\left|\tau'\left(\frac{|y|}{\delta}\right)\right|+\frac{1}{\delta^{d+2}}\left|\tau''\left(\frac{|y|}{\delta}\right)\right|\right)|\xi|^2.
\] 
Similarly to the gradient, the Hessian also integrates to zero, that is, 
$$
\int_{\R^n}\partial_{ij} \rho_\delta(y)\,\dd y = 0\quad \textup{for all} \quad i,j=1,\ldots ,d.
$$
Indeed, when $i\neq j$, the result follows by antisymmetry in $y$.  When $i=j$, we are integrating $\partial_{ii} \rho_\delta$,
which yields zero since $\partial_i \rho_\delta=0$ on $\R^d \setminus B_\delta$.
 As in the proof of the previous lemma, it  follows that 
\begin{equation}
\begin{split}
\|D^2 f_\delta\|&=\left\|\int_{B_\delta}D^2 \rho_\delta(y)\left(f(x-y)-f(x)\right) dy\right\|\leq L\int_{B_\delta}\|D^2 \rho_\delta(y)\||y|^\alpha \dd y\\
&\leq \frac{M}{\omega_d} L \delta^{\alpha}  \int_{B_\delta}\left(\frac{1}{\delta^{d+1}|y|}\left|\tau'\left(\frac{|y|}{\delta}\right)\right|+\frac{1}{\delta^{d+2}}\left|\tau''\left(\frac{|y|}{\delta}\right)\right|\right)
\dd y\\
&=ML \delta^{\alpha-2} \int_0^1\left(\frac{|\tau'(r)|}{r} +|\tau''(r)|\right)r^{d-1}\dd r,
\end{split}
\end{equation}
where $\|\cdot \|$ denotes the operator norm. 
Now, by Taylor expansion
\[
f_\delta(x \pm y)=f_\delta\phi(x)\pm \nabla f_\delta(x)\cdot y+ \frac{1}{2}\sum_{\beta=1} \frac{\partial^{|\beta|}f_\delta}{\partial x_\beta}( z^{\pm}) y^\beta.
\]
Thus,
\[
\begin{split}
|f_\delta(x + y)+f_\delta(x - y)-2f_\delta(x)|&\leq |y|^2 \| D^2 f_\delta (z) \|\\
& \leq M \left(\int_0^1\left(\frac{|\tau'(r)|}{r} +|\tau''(r)|\right)r^{d-1}\dd r \right) L \delta^{\alpha-2}|y|^2.\qedhere
\end{split}
\]
\end{proof}

\subsection{Explicit constants in dimensions one, two and three}
\label{sec:ctes}
Here we  will compute explicit constants for the mollifier that is based on the choice 
$$
\tau (r)= e^{-\frac{1}{1-r^2}}\chi_{ [0,1)}(r).
$$
\noindent {\it In one dimension:} We have
$$
M = \left(\int_0^1 \tau (r) \dd r\right)^{-1}\leq 4.51
$$
and
\[
K_1=M\int_0^1 |\tau'(r)|\dd r = M\int_0^1(- \tau'(r)) \dd r=M \tau(0)=\frac{M}{e}\leq 1.67.
\]
Since 
\[
\int_0^1 \frac{|\tau'(r)|}{r}=2\int_0^1 \frac{e^{-\frac{1}{1-r^2}}}{(1-r^2)^2}\dd r \leq 0.8,
\]
and
\[
\int_0^1 |\tau''(r)|dr=\int_0^1 \left|e^{-\frac{1}{1-r^2}}\frac{6r^4-2}{(1-r^2)^4}\right|\dd r\leq 1.6,
\]
we conclude that 
$$
K_2\leq 2.4M\leq 10.83.
$$

\noindent {\it In two dimensions:} We have
$$
M = \left(\int_0^1 \tau (r) r\dd r\right)^{-1}\leq 13.47
$$
and
\[
K_1=M\int_0^1 |\tau'(r)|r\dd r = M\int_0^1(- \tau'(r)r) \dd r=M\int_0^1\tau(r) \dd r\leq 0.23M\leq 3.13.
\]
Since 
\[
\int_0^1 |\tau'(r)| \dd r=\frac{ 1}{e}
\]
and
\[
\int_0^1 |\tau''(r)| r dr=\int_0^1 \left|e^{-\frac{1}{1-r^2}}\frac{6r^4-2}{(1-r^2)^4}r \right|\dd r\leq 1.04,
\]
we conclude that 
$$
K_2\leq M(e^{-1}+1.04)\leq  18.97.
$$
\noindent  {\it In three dimensions:} We have
$$
M = \left(\int_0^1 \tau (r) r^2\dd r\right)^{-1}\leq 28.49
$$
and
\[
K_1=M\int_0^1 |\tau'(r)|r^2\dd r = M\int_0^1(- \tau'(r)r^2) \dd r=2M\int_0^1\tau(r)r \dd r\leq 2\times 0.08 M\leq 4.56.
\]
Since 
\[
\int_0^1 |\tau'(r)|r\dd r=\int_0^1\tau(r) \dd r\leq 0.23
\]
and
\[
\int_0^1 |\tau''(r)| r^2 dr=\int_0^1 \left|e^{-\frac{1}{1-r^2}}\frac{6r^4-2}{(1-r^2)^4}r^2 \right|\dd r\leq 0.79,
\]
we conclude that 
$$
K_2\leq M(0.23+0.79)\leq 29.06.
$$
\section{Pointwise inequalities} 
 The following lemma follows from the Taylor expansion of the function $t\mapsto |t|^{p-2}t$.
\begin{lemma}{\label{lem:pineq1}}
Let $p\geq 2$. Then 
$$
\Big||a+b|^{p-2}(a+b)-|a|^{p-2}a\Big|\leq (p-1)\max( |a|,|a+b|)^{p-2}|b|.
$$
\end{lemma}
 \bigskip
{\bf Acknowledgment:} F\'elix del Teso was supported by the Spanish Government through PGC2018-094522-B-I00, and RYC2020-029589-I funded by the MICIN/AEI. Erik Lindgren was supported by the Swedish Research Council, 2017-03736.
\bigskip

\bibliographystyle{abbrv}



\end{document}